\documentclass[12pt]{article}

\usepackage{amsmath, amssymb, setspace, amsthm, fullpage}
\usepackage[all]{xy}

\input diagxy

\newcommand{\Eq}{\mathrm{Eq}}
\newcommand{\Pt}{\mathrm{Pt}}
\newcommand{\C}{\mathbb{C}}

\newcommand{\V}{\mathcal{V}}
\newcommand{\mat}[4]{\begin{pmatrix}
#1 & #3 \\
#2 & #4
\end{pmatrix}
}

\author{Michael Hoefnagel}

\title{Products and coequalizers in pointed categories}

\newtheorem{theorem}{Theorem}
\newtheorem{proposition}{Proposition}
\newtheorem{lemma}{Lemma}
\newtheorem{definition}{Definition}
\newtheorem{example}{Definition}
\begin{document}

\maketitle

\begin{abstract}
In this paper, we investigate the property (P) that finite products commute with arbitrary coequalizers in pointed categories. Examples of such categories include any regular unital or (pointed) majority category with coequalizers, as well as any pointed factor permutable category with coequalizers. We establish a Mal'tsev term condition characterizing pointed varieties of universal algebras satisfying (P). We then consider categories satisfying (P) locally, i.e., those categories for which every fibre $\Pt_{\C}(X)$ of the fibration of points $\pi:\Pt(\C) \rightarrow \C$ satisfies (P). Examples include any regular Mal'tsev or majority category with coequalizers, as well as any regular Gumm category with coequalizers. Varieties satisfying (P) locally are also characterized by a Mal'tsev term condition, which turns out to be equivalent to a variant of Gumm's shifting lemma. Furthermore, we show that the varieties satisfying (P) locally are precisely the varieties with normal local projections in the sense of Z. Janelidze. 
\end{abstract}

\section{Introduction}
Consider a variety $\C$ of universal algebras in which finite products commute with arbitrary coequalizers. That is, $\C$ is a variety which satisfies the following property: 
\begin{itemize}
\item[(P)] For any two coequalizer diagrams
\[
    \xymatrix{
    C_1 \ar@<-.5ex>[r]_{u} \ar@<.5ex>[r]^{v} & X \ar[r]^{q_1} & X'  & C_2 \ar@<-.5ex>[r]_{u'} \ar@<.5ex>[r]^{v'}& Y \ar[r]^{q_2} & Y',
    }
\]
in $\C$, the diagram 
\[
\xymatrix{ 
& C_1 \times C_2 \ar@<-.5ex>[r]_-{u\times u'} \ar@<.5ex>[r]^-{v \times v'}  & X \times Y \ar[r]^-{q_1 \times q_2} & X' \times Y', \\ 
}
\]
is a coequalizer diagram in $\C$.
\end{itemize}
 It is immediate that $\C$ must have constants, since if it did not, then the empty set would be the initial object $0$ in $\C$, and applying (P) to the two diagrams 
 \[
    \xymatrix{
    X \times X  \ar@<-.5ex>[r]_-{\pi_1} \ar@<.5ex>[r]^-{\pi_2} & X \ar[r] & 1   & 0 \ar@<-.5ex>[r]_{0} \ar@<.5ex>[r]^{0}& X \ar[r]^{1_X} & X
    }
\] 
would imply that
 \[
\xymatrix{
0 \ar@<-.5ex>[r]_-{0} \ar@<.5ex>[r]^-{0} & X \times X \ar[r]^-{\pi_1} & X
}
\]
 is a coequalizer for any algebra $X$. But this would force the variety to be trivial, so that $\C$ must possess constants. In this paper, we are concerned with varieties which possess a unique constant, i.e., \emph{pointed varieties}. We show that a pointed variety satisfies (P) if and only if it admits binary terms $b_{i}(x,y)$ and unary terms $c_{i}(x)$ for each $1 \leqslant i \leqslant m$ and $(m+2)$-ary terms $p_1,p_2,...,p_n$ satisfying the equations: 
\begin{align*}
p_1(x,y,b_{1}(x,y),b_{2}(x,y),...,b_{m}(x,y)) &= x, \\
p_i(y,x,b_{1}(x,y),...,b_{m}(x,y)) &= p_{i+1}(x,y,b_{1}(x,y),...,b_{m}(x,y)),\\
p_n(y,x,&b_{1}(x,y),b_{2}(x,y),...,b_{m}(x,y)) = y,  \\
\end{align*}
and for each $i = 1,...,n$ we have
\begin{align*}
p_i(0,0,c_1(z),...,c_{m}(z)) &= z.   \\
\end{align*}
We then consider varieties which satisfy (P) \emph{locally}, i.e., varieties for which each fibre $\Pt_{\C}(X)$ of the \emph{fibration of points} $\pi: \Pt(\C) \rightarrow \C$ satisfies (P). Every pointed variety $\C$ satisfying (P) necessarily has \emph{normal projections} in the sense of \cite{Janelidze2003}, i.e., every product projection in $\C$ is a cokernel, but the converse is not true. However, it turns out that a variety satisfies (P) locally if and only if it has \emph{local normal projections} \cite{Janelidze04} (see Theorem~\ref{Thm:local-product-coequalizer}). Furthermore, we show how both (P) and its local version may be seen as variants of Gumm's shifting lemma \cite{Gum83}, so that in particular any congruence modular variety of universal algebras satisfies (P) locally. 

\section{Main Results}
Recall from \cite{Janelidze2003} that a pointed category $\C$ has normal projections if any product projection in $\C$ is a normal epimorphism. For example, every \emph{subtractive} category \cite{Janelidze2005}, as well as any \emph{unital} category \cite{Bou02}, have normal projections. 
\begin{proposition} \label{Prop:P-NPP}
If $\C$ is a pointed category with binary products which satisfies (P), then the product of two normal epimorphisms is normal, and in particular $\C$ has normal projections. 
\end{proposition}
\begin{proof}
The product of two normal epimorphisms being normal is a trivial consequence of (P), and any product projection $\pi_1:X \times Y \rightarrow X$ may be obtained as the product of $X \xrightarrow{1_X} X$ and $Y \rightarrow 0$, which are normal epimorphisms.
\end{proof}
\noindent
The lemma below has a straightforward proof, and therefore we leave the proof of it to the reader. 
\begin{lemma}\label{Lem:coequalizer}
In a category, given morphisms
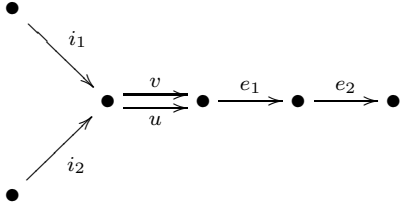
\[
\xymatrix{ \bullet\ar[dr]^-{i_1} & & & & \\ 
& \bullet \ar@<-.5ex>[r]_{u} \ar@<.5ex>[r]^{v}  & \bullet\ar[r]^-{e_1} & \bullet\ar[r]^-{e_2} & \bullet 
\\ \bullet\ar[ur]_-{i_2} & & & & }
\]
such that $e_2\circ e_1 \circ u = e_2\circ e_1 \circ v $, where $e_1$ is a coequalizer of $u \circ  \iota_1, v \circ \iota_1$ and $e_2$ is a coequalizer of $e_1 \circ  u \circ i_2, e_1 \circ  v \circ i_2 $, then $e_2\circ e_1$ is a coequalizer of $v,u$.
\end{lemma}
\noindent
In what follows, the $0$ symbol is used to denote an initial object in a category $\C$ (if it exists), and $0_X$ denotes the unique morphism $0_X: 0 \rightarrow X$. 
\begin{lemma}\label{Lem:product-coequalizer}
The following are equivalent for a category with binary products and an initial object.
\begin{enumerate}
\item $\C$ satisfies (P), i.e., the product of two coequalizers is a coequalizer.
\item The product of any coequalizer diagram $C \rightrightarrows X \rightarrow Y$ with the trivial coequalizer 
\[
\xymatrix{
0 \ar@<-.5ex>[r]_{0_X} \ar@<.5ex>[r]^{0_X}& Y \ar[r]^{1_{Y}} & Y,
}
\]
is a coequalizer diagram.
\end{enumerate}
\end{lemma}
\begin{proof}
We show that (2) implies (1). Suppose that 
\[
    \xymatrix{
    C_1 \ar@<-.5ex>[r]_{u} \ar@<.5ex>[r]^{v} & X \ar[r]^{q_1} & X'   & C_2 \ar@<-.5ex>[r]_{u'} \ar@<.5ex>[r]^{v'}& Y \ar[r]^{q_2} & Y'
    }
\]
are two coequalizer diagrams in $\C$. Then applying (2), we have that the diagrams 
\[
    \xymatrix{
    C_1 \times 0 \ar@<-.5ex>[r]_-{u \times 0_Y} \ar@<.5ex>[r]^-{v \times 0_Y} & X \times Y \ar[r]^{q_1 \times 1_Y} & X' \times Y   & 0 \times C_2 \ar@<-.5ex>[r]_-{0_{X'} \times u'} \ar@<.5ex>[r]^-{0_{X'} \times v'} & X' \times Y \ar[r]^{1_{X'} \times q_2} & X' \times Y',
    }
\]
are coequalizer diagrams. We may then apply Lemma~\ref{Lem:coequalizer} to the diagram
\[
\xymatrix{ C_1 \times 0 \ar[dr]^-{1_{C_1} \times 0_{C_2}} & & & & \\ 
& C_1 \times C_2 \ar@<-.5ex>[r]_-{u\times u'} \ar@<.5ex>[r]^-{v \times v'}  & X \times Y \ar[r]^-{q_1 \times 1_Y} &  X' \times Y \ar[r]^-{1_{A'} \times q_2} & X' \times Y' \\ 
0\times C_2 \ar[ur]_-{0_{C_1} \times 1_{C_2}} & & & & }
\]
to obtain $q_1 \times 1_Y \circ 1_{A'} \times q_2 = q_1 \times q_2$ as a coequalizer of $u\times u'$ and $v \times v'$.
\end{proof}
In what follows, we will be working with pointed categories. To simplify the notation, we will always denote zero-morphisms between objects by $0$, when there is no ambiguity.
\begin{proposition} \label{Prop:prod-coeq-egg-box}
Let $\C$ be a pointed category with finite limits and coequalizers, then the following are equivalent. 
\begin{enumerate}
\item $\C$ satisfies (P).
\item The product of a regular epimorphism in $\C$ with an isomorphism in $\C$ is a regular epimorphism, and for any effective equivalence relation $C$ on any product $A \times B$ in $\C$, we have that $(x,0)C(y,0)$ implies $(x,z) C (y,z)$ for any generalized elements $x,y:S \rightarrow A$ and $z: S \rightarrow B$. This property is illustrated by the diagram below, which can be seen as a variant of the ``egg-box property'' in the sense of \cite{Chajda1988}.
\[
\xymatrix{
(x,0) \ar@/^1.5pc/@{-}[r]^{C} \ar@{-}[r]^{\Eq(\pi_2)} \ar@{-}[d]_{\Eq(\pi_1)} & (y,0) \ar@{-}[d]^{\Eq(\pi_1)} \\
(x,z) \ar@/_1.5pc/@{..}[r]_{C} \ar@{-}[r]_{\Eq(\pi_2)} & (y,z) 
}
\]
\end{enumerate}
\end{proposition}
\begin{proof}
For (1) $\implies$ (2), let $C$ be any effective equivalence relation on a product $A \times B$ in $\C$, and let $x,y:S \rightarrow A$ and $z:S \rightarrow B$ be any morphisms such that $(x,0) C (y,0)$. Let $q:A \rightarrow Q$ be a coequalizer of $x$ and $y$, then $q \times 1_B$ is a coequalizer of $(x,0)$ and $(y,0)$. Since $C$ is effective and $(x,0) C (y,0)$, it follows that $\Eq(q \times 1_B) \leqslant C$ and hence $(x,z) C (y,z)$. Note that the product of a regular epimorphism in $\C$ with an isomorphism in $\C$ being regular is an immediate consequence of (P). 

For (2) $\implies$ (1), suppose that we are given a coequalizer $q:A \rightarrow Q$ of $a:T \rightarrow A$ and $b:T \rightarrow A$, and let $C = \Eq(q')$ be the kernel equivalence relation of a coequalizer $q':A \times B \rightarrow Q'$ of $(a,0)$ and $(b,0)$. It suffices to show that $\Eq(q \times 1_B) = C$, since $q \times 1_B$ is regular by assumption. We always have $C \leqslant \Eq(q \times 1_B)$, and given any generalized elements $x,y:S \rightarrow A$ and $z:S \rightarrow B$, if $(x,z)\Eq(q \times 1_B)(y,z)$ then $(x,0)C(y,0)$, and hence we may apply (2) to get $(x,z)C(y,z)$, which shows that $\Eq(q \times 1_B) \leqslant C$.
\end{proof}

In what follows we will fix a finitely complete pointed category $\C$ which has coequalizers, and we will assume that regular epimorphisms in $\C$ are stable under binary products. Recall that $\C$ is \emph{unital} if for any binary relation $R$ in $\C$ between any two objects $A$ and $B$ in $\C$, we have $a R 0$ and $0 R b$ implies $a R b$ (see \cite{Bou02, ZJan06}). 

\begin{proposition} \label{Prop:unital-categories-satsify-P}
If $\C$ is unital, then $\C$ satisfies (P).
\end{proposition}
\begin{proof}
We show that $\C$ fulfills condition (2) of Proposition~\ref{Prop:prod-coeq-egg-box}. Let $C$ be any effective equivalence relation on any product $A \times B$ in $\C$ such that $(x,0) C (y,0)$. Then we may consider the binary relation $R$ between $A\times A$ and $B\times B$ defined by $(a,a') R (b,b')$ if and only if $(a,b) C (a',b')$. By assumption we have $(x,y) R (0,0)$ and $(0,0) R (z,z)$ (by reflexivity of $C$) so that $(x,y) R (z,z)$ and hence $(x,z) C (y,z)$. 
\end{proof}
Recall that $\C$ is a \emph{majority category} \cite{Hoe18a, Hoefnagel2019} if for any ternary relation $R$ between objects $X,Y,Z$ we have 
\[
(x,y,z') \in R \quad \text{and} \quad (x,y',z) \in R \quad \text{and} \quad (x',y,z) \in R \quad \implies  (x,y,z)\in R,\tag{$*$}
\]
for any generalized elements $x,x':S \rightarrow X$ and $y,y':S \rightarrow X$ and $z,z':S \rightarrow X$ in $\C$.
\begin{proposition} \label{Prop:majority-categories-satsify-P}
If $\C$ is a majority category, then $\C$ satisfies (P).
\end{proposition}
\begin{proof}
Let $C$ be any effective equivalence relation on $A \times B$ and suppose that $x,y,z$ are as in the statement of Proposition~\ref{Prop:prod-coeq-egg-box}. Then we consider the ternary relation $R$ between $A, B$ and $A \times B$ defined by $(a,b, (a',b')) \in R$ if and only if $(a,a') C (b,b')$. Then by assumption we have $(x,y, (0,0)) \in R$, and by reflexivity we have $(x,x,(z,z)) \in R$ and $(y,y,(z,z)) \in R$. Applying the majority property ($*$) above to these three elements yields $(x,y,(z,z)) \in R$, so that $(x,z) C (y,z)$. 
\end{proof}
The notion of a \emph{Gumm} category \cite{DBMG2004} is the categorical analogue of varieties in which Gumm's shifting lemma holds \cite{Gum83}, i.e., congruence modular varieties. A finitely complete category $\C$ is a Gumm category if for any three equivalence relations $R,S,T$ on any object $X$ in $\C$ such that $R \cap S \leqslant T$, if $(x,y),(w,z) \in R$ and $(y,z), (x,w) \in S$ and $(y,z) \in T$ then we get $(x,w) \in T$. This implication of relations between the elements above is usually depicted with a diagram
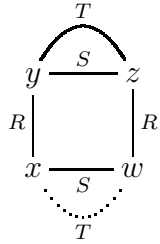
\[
\xymatrix{
y \ar@/^1.5pc/@{-}[r]^{T} \ar@{-}[r]^{S} \ar@{-}[d]_{R} & z \ar@{-}[d]^{R} \\
x\ar@/_1.5pc/@{..}[r]_{T} \ar@{-}[r]_{S} & w 
}
\]
where the dotted curve represents the relation induced from the relations indicated by the solid curves. 
\begin{proposition} \label{Prop:Gumm-categories-satsify-P}
If $\C$ is a Gumm category, then $\C$ satisfies (P).
\end{proposition}
\begin{proof}
The diagrammatic condition characterizing (P) in (2) of Proposition~\ref{Prop:prod-coeq-egg-box} is a restriction of the shifting lemma, where $R = \Eq(\pi_1)$, $S = \Eq(\pi_2)$ and $T = C$. Since we always have $\Eq(\pi_1) \cap \Eq(\pi_2) = \Delta_X$ for any two complementary product projections $\pi_1$ and $\pi_2$ of an object $X$ in $\C$, we may apply the shifting lemma to the diagram in Proposition~\ref{Prop:prod-coeq-egg-box} so that $\C$ satisfies (P).
\end{proof}
\begin{definition}[\cite{Gra04}]
A regular category $\C$ is said to be \emph{factor permutable} if $F \circ E = E \circ F$ for any factor relation $F$ and any equivalence relation $E$ on any object $X$ in $\C$. 
\end{definition}
\begin{lemma}[Lemma 2.5 in \cite{Gra04}] \label{Lem:weak-shifting}
In any factor permutable category $\C$ the \emph{weak shifting lemma holds}: for any equivalence relations $R$ and $S$ on $A \times B$ in $\C$ such that $\Eq(\pi_1) \cap R \leqslant S$, if $(a,b), (a,c), (d,e),(d,f)$ are related via the solid arrows as in the diagram 

\[
\xymatrix{
(a,c) \ar@/^1.5pc/@{-}[r]^{S} \ar@{-}[r]^{R} \ar@{-}[d]_{\Eq(\pi_1)} & (d,f) \ar@{-}[d]^{\Eq(\pi_1)} \\
(a,b)\ar@/_1.5pc/@{..}[r]_{S} \ar@{-}[r]_{R} & (d,e) 
}
\]
then we have $(a,b) S (d,e)$.
\end{lemma}
\begin{proposition}
Any pointed regular factor permutable category with coequalizers satisfies (P).
\end{proposition}
\begin{proof}
Similar to the proof of Proposition~\ref{Prop:Gumm-categories-satsify-P}, we may apply the weak shifting property of Lemma~\ref{Lem:weak-shifting} to the diagram in (2) of Proposition~\ref{Prop:prod-coeq-egg-box}. 
\end{proof}
The theorem below is a Mal'tsev type characterization of pointed varieties of universal algebras satisfying (P). In the proof, we will use 2$\times$2 matrices to represent elements of a congruence $C$ on a product $A \times B$ in the following way:
\[
\mat{a}{b}{a'}{b'} \in C \Longleftrightarrow (a,b) C (a',b').
\]
\begin{theorem} \label{Thm:product-coequalizer-variety}
A pointed variety $\V$ of algebras satisfies (P) if and only if $\V$ admits binary terms $b_{i}(x,y)$ and unary terms $c_{i}(x)$ for each $1 \leqslant i \leqslant m$ and $(m+2)$-ary terms $p_1,p_2,...,p_n$ satisfying the equations: 
\begin{align*}
p_1(x,y,b_{1}(x,y),b_{2}(x,y),...,b_{m}(x,y)) &= x, \\
p_i(y,x,b_{1}(x,y),...,b_{m}(x,y)) &= p_{i+1}(x,y,b_{1}(x,y),...,b_{m}(x,y)),\\
p_n(y,x,&b_{1}(x,y),b_{2}(x,y),...,b_{m}(x,y)) = y,  \\
\end{align*}
and for each $i = 1,...,n$ we have
\begin{align*}
p_i(0,0,c_1(z),...,c_{m}(z)) &= z.   \\
\end{align*}  
\end{theorem}
\begin{proof}
Consider the principal congruence $C$ on $F_{\V}(x,y) \times F_{\V}(z)$ generated by the single relation $(x,0)C(y,0)$, where $F_{\V}(x,y)$ and $F_{\V}(z)$ are the free algebras over $\{x,y\}$ and $\{z\}$ respectively. By (2) of  Proposition~\ref{Prop:prod-coeq-egg-box} it follows that $(x,z) C (y,z)$. The congruence $C$ may be obtained by closing the relation
\[
\big\{\mat{x}{0}{y}{0}, \mat{y}{0}{x}{0}\big\}
\]
first under reflexivity, then under all operations in $\V$, and then under transitivity. Therefore, there exists $z_1,...,z_n \in A \times B$ such that $z_1 = (x,z)$ and $z_n = (y,z)$, where:
\[
(z_{i}, z_{i+1}) = p_i(\mat{x}{0}{y}{0}, \mat{y}{0}{x}{0}, \mat{b_{1,i}(x,y)}{c_{1,i}(z)}{b_{1,i}(x,y)}{c_{1,i}(z)}, \cdots,\mat{b_{m_i,i}(x,y)}{c_{m_i,i}(z)}{b_{m_i,i}(x,y)}{c_{m_i,i}(z)}).
\]
for certain terms $p_1,\dots,p_n$ and elements $(b_{ij}(x,y),c_{ij}(z)\in \mathsf{Fr}\{x,y\}\times \mathsf{Fr}\{z\}$. Without loss of generality, we may assume that $m_1=\dots=m_n=m$, $b_{1j}=\dots=b_{nj}=b_j$ and $c_{1j}=\dots=c_{nj}=c_j$ for all $j\in\{1,\dots,m\}$. Then, writing out the identities above coordinate-wise and noting that since $(b_j,c_j)\in \mathsf{Fr}\{x,y\}\times \mathsf{Fr}\{z\}$, each $b_j$ is a binary term $b_j(x,y)$ and each $c_j$ is a unary term $c_j(z)$, we get the identities in the statement of the theorem.  

For the converse, suppose that $C$ is any congruence on $X \times Y$ in $\V$, and that $(x,0) C (y,0)$. Consider the elements of $C$ defined by:
\[
(z_{1,i}, z_{2, i}) = p_i(\mat{x}{0}{y}{0}, \mat{y}{0}{x}{0}, \mat{b_1(x,y)}{c_1(z)}{b_1(x,y)}{c_1(z)}, \cdots,\mat{b_m(x,y)}{c_m(z)}{b_m(x,y)}{c_m(z)}).
\]
Then the equations at (2) imply that $z_{2,i} = z_{1, i+1}$ as well as $(x,z) = z_{1,1}$ and $(y,z) = z_{2,n}$, so that by the transitivity of $C$ we get $(x,z) C (y,z)$. 
\end{proof}

\section{Local normal projections and products of coequalizers}
Recall that the \emph{category of points} $\Pt_{\C}(X)$ of an object $X$ in a category $\C$ consists of triples $(A,p,s)$ where $p:A \rightarrow X$ is a split epimorphism and $s$ is a splitting for $p$. A morphism $f:(A,p,s) \rightarrow (B,q,t)$ in $\Pt_{\C}(X)$ is a morphism $f:A \rightarrow B$ such that $q \circ f = p$ and $f \circ s = t$. The category $\Pt_{\C}(X)$ is pointed, where the zero-object is $(X,1_X, 1_X)$, and if $\C$ is finitely complete, then so is $\Pt_{\C}(X)$. The zero-morphism from $(A,p,s)$ to $(B,q,t)$ is given by $t \circ p$. When $\C$ has coequalizers and pullbacks, then $\Pt_{\C}(X)$ has products and coequalizers. Moreover, if $\C$ is a regular category, then so is $\Pt_{\C}(X)$. 

In what follows, we will say that a category $\C$ \emph{satisfies (P) locally} if for every object $X$ in $\C$ the category $\Pt_{\C}(X)$ satisfies (P). 

\begin{example}
Every regular Mal'tsev category \cite{CPP91, CLP91} with coequalizers satisfies (P) locally. This is because a finitely complete category $\C$ is Mal'tsev if and only if for any object $X$ the category $\Pt_{\C}(X)$ is unital (see \cite{Bou96}). Moreover, $\C$ being regular with coequalizers implies that $\Pt_{\C}(X)$ is pointed regular with coequalizers. Then by Proposition~\ref{Prop:unital-categories-satsify-P} it follows that $\Pt_{\C}(X)$ satisfies (P), for any object $X$ in $\C$. 
\end{example}

\begin{example}
For essentially the same reasons as the above example every regular majority category with coequalizers satisfies (P) locally: if $\C$ is a regular majority category with coequalizers, then so is $\Pt_{\C}(X)$ for any object $X$, which is moreover pointed. Hence by Proposition~\ref{Prop:majority-categories-satsify-P} we have that $\C$ satisfies (P) locally. 
\end{example}

\begin{example}
Every regular  Gumm category $\C$ with coequalizers satisfies (P) locally: if $\C$ is a Gumm category with coequalizers, then so is $\Pt_{\C}(X)$ (see Lemma~2.5 in \cite{DBMG2004} and the discussion proceeding it), and hence by Proposition~\ref{Prop:Gumm-categories-satsify-P} it follows that $\Pt_{\C}(X)$ satisfies (P) for any object $X$ in $\C$. 
\end{example}

Recall from the introduction that a category $\C$ is said to have normal local projections if for any object $X$ in $\C$, the category $\Pt_{\C}(X)$ has normal projections.  
\begin{theorem}[\cite{Janelidze04}] \label{Thm:Janelidze-2004}
The following are equivalent for a variety $\V$ of universal algebras. 
\begin{itemize}
\item $\V$ has normal local projections. 
    \item $\V$ admits binary terms $b_1,...,b_m, c_1,...,c_m$ and $(m+2)$-ary terms $p_1,p_2,...,p_n$ such that
    \begin{itemize}
        \item $p_0(x,y,b_1(x,y),...,b_m(x,y)) = x$.
        \item $p_n(y,x,b_1(x,y),...,b_m(x,y)) = y$.
        \item $p_{i+1}(y,x,b_1(x,y),...,b_m(x,y)) = p_i(x,y,b_1(x,y),...,b_m(x,y))$.
        \item For any $i \in \{1,2,...,n\}$ we have  $p_i(u,u,c_1(u,v),...,c_m(u,v)) = v$ and $b_i(z,z) = c_i(z,z)$. 
    \end{itemize}
\end{itemize}
\end{theorem}
In what follows we will see that a variety of universal algebras satisfies (P) locally if and only if it has local normal projections. 
\subsection{Characterization of varieties satisfying (P) locally}
Given two objects $(A,p,s)$ and $(B,q,t)$ in $\Pt_{\C}(X)$ consider the diagram below where the square is a pullback 
\[
\xymatrix{
& X \ar[rd]|{(s,t)}\ar@/^1pc/@{->}[rrd]^{t} \ar@/_1pc/@{->}[rdd]_{s} & & \\
& & A \times_X B \ar[r]^{p_2} \ar[d]_{p_1} \ar[dr]|d & B \ar[d]^{q} \\
& & A \ar[r]_{p} & X
}
\]
Then $(A \times_X B, d, (s,t))$ together with $p_1,p_2$ form a product for $(A,p,s)$ and $(B,q,t)$ in $\Pt_{\C}(X)$. Then we may adapt Lemma~\ref{Lem:product-coequalizer} to the local situation, and obtain the following: 
\begin{proposition}
 For a category $\C$ with finite limits, the following are equivalent.
\begin{itemize}
    \item $\C$ satisfies (P) locally. 
    \item For any object $X$ in $\C$, and for any coequalizer diagram
    \[
\xymatrix{
(C, r, n) \ar@<-.5ex>[r]_-{u} \ar@<.5ex>[r]^-{v} & (A,p,s) \ar[r]^{f} & (B,q,t)
}
\]
in $\Pt_{\C}(X)$, the diagram 
    \[
\xymatrix{
C\ar@<-.5ex>[r]_-{(u, t \circ r)} \ar@<.5ex>[r]^-{(v, t \circ r)} & A\times_X B \ar[r]^{f \times_X 1_B} & B \times_X B
}
\]
is a coequalizer in $\C$. 
\end{itemize}
\end{proposition}

\begin{theorem} \label{Thm:local-product-coequalizer}
For a variety $\V$ the following are equivalent.
\begin{enumerate}
    \item For any object $X$ in $\V$ the product of two coequalizer diagrams in $\Pt_{\V}(X)$ is a coequalizer diagram. 
    \item For any object $X$ in $\V$ the product of two normal-epimorphisms $\Pt_{\V}(X)$ is a normal epimorphism.
    \item $\V$ has normal local projections. 
    \item $\V$ admits binary terms $b_1,...b_m, c_1,...,c_m$ and $(m+2)$-ary terms $p_1,p_2,...,p_n$ such that
    \begin{itemize}
        \item $p_0(x,y,b_1(x,y),...,b_m(x,y)) = x$.
        \item $p_n(y,x,b_1(x,y),...,b_m(x,y)) = y$.
        \item $p_{i+1}(y,x,b_1(x,y),...,b_m(x,y)) = p_i(x,y,b_1(x,y),...,b_m(x,y))$.
        \item For any $i \in \{1,2,...,n\}$ we have  $p_i(u,u,c_1(u,v),...,c_m(u,v)) = v$ and $b_i(z,z) = c_i(z,z)$. 
    \end{itemize}
    \item For any two morphisms $f:A \rightarrow X$ and $g:B \rightarrow X$, and any congruence $C$ on the pullback $A \times_X B$ of $f$ along $g$, if $(x,u) C (y,u)$ then $(x,v) C (y,v)$ where $(x,v), (y,v) \in A \times_X B$. This can be visualized as follows
    \[
\xymatrix{
(x,u) \ar@/^1.5pc/@{-}[r]^{C} \ar@{-}[r]^{\Eq(p_2)} \ar@{-}[d]_{\Eq(p_1)} & (y,u) \ar@{-}[d]^{\Eq(p_1)} \\
(x,v) \ar@/_1.5pc/@{..}[r]_{C} \ar@{-}[r]_{\Eq(p_2)} & (y,v) 
}
\]
\end{enumerate}
\end{theorem}
\noindent
In the proof below, we make use of the same notation as described in the paragraph preceding  Theorem~\ref{Thm:product-coequalizer-variety}.
\begin{proof}
The implications (1) $\implies$ (2) $\implies$ (3) are the content of Proposition~\ref{Prop:P-NPP}, and (3) $\implies$ (4) is the content of Theorem~\ref{Thm:Janelidze-2004}. We show (4) $\implies$ (5) $\implies$ (1). Suppose that $C$ is a congruence on the pullback $A \times_X B$ in (5), and that we are given $(x, u) C (y, u)$, and suppose that $(x,v), (y,v) \in A \times_X B$ are any two elements. Note that since $(x, u), (y, u),(x,v), (y,v)$ are all elements of $A \times_X B$, it follows that $f(x) = f(y) = g(u) = g(v)$, and since $b_i(z,z) = c_i(z,z)$ we have  $f(b_i(x,y)) = b_i(f(x), f(y)) = c_i(g(u),g(v)) = g(c_i(u,v))$, so that $(b_i(x,y), c_i(u,v)) \in A \times_X B$. Consider the elements $(z_{1,i}, z_{2, i})$ of $C$ given by:
\[
(z_{1,i}, z_{2, i}) =
p_i \bigg(\mat{x}{u}{y}{u}, \mat{y}{u}{x}{u}, \mat{b_1(x,y)}{c_1(u,v)}{b_1(x,y)}{c_1(u,v)}, \cdots,\mat{b_m(x,y)}{c_m(u,v)}{b_m(x,y)}{c_m(u,v)} \bigg).
\]
Then the equations in the statement of the theorem imply that $z_{1,0} = (x,v)$ and $z_{2,n} = (y,v)$, and moreover that $z_{2,i} = z_{1,i+1}$, so that by the transitivity of $C$ we have $(x,v)C(y,v)$. 
For (5) $\implies$ (1), we note that condition (5) implies that $\Pt_{\V}(X)$ satisfies (2) of Proposition~\ref{Prop:prod-coeq-egg-box}, since products in $\Pt_{\V}(X)$ are pullbacks, and the domain functor $\Pt_{\V}(X) \rightarrow \V$ sends equivalence relations in $\Pt_{\V}(X)$ to equivalence relations in $\V$ (since it preserves pullbacks and reflects isomorphisms).
\end{proof}
In light of the theorem above, it is natural to ask if every variety with normal projections satisfies (P). We answer this question in the negative: every subtractive category with finite limits has normal product projections, but not every subtractive variety satisfies (P). Moreover, the product of two normal epimorphisms in a subtractive variety need not be normal, and thus cannot satisfy (P) by Proposition~\ref{Prop:P-NPP}. To see this, consider the subtraction algebra $X$ which has underlying set $\{0,a,b\}$, and whose subtraction is defined as $x - y = x$ if $y = 0$, and $x- y =0$ otherwise. Then consider the subtraction algebra $Y = \{c, 0\}$ whose subtraction is defined as $X$'s is. The function $X \xrightarrow{f} Y$ defined by $f(a) = 0$ and $f(b) = z$ is a normal epimorphism, whose kernel is $\ker(f)= \{0,a\}$. If $f \times 1_Y$ were normal, then $\Eq(f \times 1_Y)$ would be the congruence $C$ generated by $(a,0) \sim (0,0)$. It is then routine to verify that $((a,c),(b,c))$ is not contained in $C$, although it is contained in $\Eq(f \times 1_Y)$.

\section{Concluding remarks}
We have been unable to establish a categorical counterpart of Theorem~\ref{Thm:local-product-coequalizer}, and leave the investigation of this question for a future work. Moreover, for general varieties we have been unable to determine whether or not the property (P) can be characterized by a Mal'tsev condition as we did in Theorem~\ref{Thm:product-coequalizer-variety}. We leave it as an open question of whether or not (P) is a Mal'tsev property for general varieties, and conjecture that it is not. We mention here that it can be shown that every  variety of algebras with constants in which every object is $\mathcal{M}$-coextensive in the sense of \cite{Hoefnagel2019c}, where $\mathcal{M}$ is the class of regular epimorphisms in $\C$, satisfies (P). This is because such varieties are exactly those that have directly decomposable congruences. Thus for example, the category $\mathbf{Ring}$ of unitary rings satisfies (P), but $\mathbf{Ring}$ is not pointed.
\section{Acknowledgements}
This work grew out of a joint work in progress with Professor Zurab Janelidze. I would like to thank him for his valuable input into the work presented here.
\bibliographystyle{plain}
\bibliography{References}

\end{document}